\documentclass{amsart}

\numberwithin{equation}{section}

\usepackage{amssymb}
\usepackage{enumerate, xspace}
\usepackage{marvosym} 
\usepackage{bbm}       
\usepackage[normalem]{ulem}   
\usepackage{changebar}
\usepackage[backgroundcolor=blue!20!white, linecolor=blue!20!white,
textsize=footnotesize]{todonotes}
\usepackage[colorlinks]{hyperref}

\allowdisplaybreaks

\newtheorem{thm}{Theorem}
\newtheorem{lem}[thm]{Lemma}
\newtheorem{sub-lem}[thm]{Sub-Lemma}

\newtheorem{rem}[thm]{Remark}

\newcommand{\bi}{{\boldsymbol i}}

\newcommand{\cA}{{\mathcal A}}

\newcommand{\cC}{{\mathcal C}}

\newcommand{\cO}{{\mathcal O}}

\newcommand{\cV}{{\mathcal V}}

\newcommand\cW{{\mathcal W}}

\newcommand{\bC}{{\mathbb C}}

\newcommand{\bH}{{\mathbb H}}
\newcommand{\bN}{{\mathbb N}}

\newcommand{\bR}{{\mathbb R}}

\newcommand{\bU}{{\mathbb U}}

\newcommand{\bZ}{{\mathbb Z}}

\newcommand{\bc}{{\mathbbm c}}
\newcommand{\be}{{\mathbbm e}}
\newcommand{\bg}{{\mathbbm g}}

\newcommand{\vf}{\varphi}

\newcommand{\frkg}{\mathfrak g}

\newcommand{\pqnorm}[2][p,q,\ell]{\left\|#2 \right\|_{#1}}

\newcommand\Id{{\mathbbm{1}}}
\newcommand\xu{\upsilon}

\newcommand\htop{h_{\text{\tiny top}}(\phi_1)}

\newcommand{\Const}{C_\#}
\newcommand{\const}{c_\#}

\newcommand\Cnd{C_1}

\newcommand\Cnc{C_6}

\begin{document}

\title[Errata]{Anosov Flows and Dynamical Zeta Functions\\ (Errata)}
\author{Paolo Giulietti}
\author{Carlangelo Liverani}
\author{Mark Pollicott}
\address{Paolo Giulietti\\
Dipartimento di Matematica\\
Universit\`a di Pisa\\
Largo Bruno Pontecorvo 5, 56127 Pisa, Italy.}
\email{{\tt paolo.giulietti@unipi.it}}
\address{Carlangelo Liverani\\
Dipartimento di Matematica\\
II Universit\`{a} di Roma (Tor Vergata)\\
Via della Ricerca Scientifica, 00133 Roma, Italy.}
\email{{\tt liverani@mat.uniroma2.it}}
\address{Mark Pollicott\\
Department of Mathematics\\
Warwick University\\
Coventry, CV4 7AL, England}
\email{{\tt masdbl@warwick.ac.uk}}
\date{\today}
\begin{abstract}
This errata fixes a mistake in the part of \cite{GLP} which proves a spectral gap for contact Anosov flows with respect to the measure of maximal entropy  (\cite[Section 7]{GLP}).
However, the first part of \cite{GLP}, in which it is proved  that the Ruelle zeta function is meromorphic, is unaffected.
\end{abstract}
\thanks{ We thank Sebasti\'en Go\"uezel for pointing out the mistake and for very helpful discussions.}
\maketitle

\section{the mistake}

Equation \cite[Equation (7.14)]{GLP} is wrong, since it does not take into account the factor $e^{z\tau_W\circ H_{\beta,{\bi},W}}$ from \cite[Equation (7.10)]{GLP} and estimates incorrectly the norm of $\vf_{k,\beta,\bi}\circ H_{\beta, \bi,W}$. The correct version of  \cite[Equation (7.14)]{GLP} is (see \eqref{eq:g1} below):\\
 let $d_s$ be the dimension of the stable manifolds, then for each $\eta\in (0,1)$,\footnote{ See below for the definition of $\lambda_+,\lambda_-,\varpi'$.}
\begin{equation}\label{eq:g-est-correct}
\|e^{z\tau_W\circ H_{\beta,{\bi},W}}\hat g_{k, \beta, \bi, W}\|_{\Gamma^{d_s,\varpi'}_c(\widetilde W_{\beta,\bi})}\leq \Const (r^{-1} +|z|)\frac{(k r)^{n-1}e^{-a k r}}{(n-1)!}\|g\|_{\Gamma^{d_s,1+\eta}_c(W)}.
\end{equation}
Unfortunately, this weaker estimate does not suffice to carry out the proof of \cite[Proposition 7.5]{GLP} as presented in \cite{GLP} (i.e. using \cite[Equation (7.16)]{GLP}).

\section{Correction}
Nevertheless, \cite[Theorem 2.4]{GLP} holds under a stronger assumption (namely some homogeneity), as  we shall show here. In particular, it applies to small perturbations of constant curvature geodesic flows in any dimension. To simplify the argument we did not try to optimise the estimate of the size of the perturbation. Before stating the correct results we need to recall and introduce some notation.

Let $C_0,c_0>0$, and $\lambda_+(x,t),\lambda_-(x,t)>0$ be such that, for each $x\in M$ and $t>0$, $\sup_{v\in E^s(x)}\frac{\|D_x\phi_{-t}v\|}{\|v\|}\leq C_0 e^{\lambda_+ (x,t) }$ and $\inf_{v\in E^s(x)}\frac{\|D_x\phi_{-t}v\|}{\|v\|}\geq c_0e^{\lambda_- (x,t)}$. Also, let
\[
\begin{split}
&\lambda_+(t)=\sup_{x\in M}\lambda_+(x,t)\;; \quad \lambda_-(t)=\inf_{x\in M}\lambda_-(x,t).
\end{split}
\]
and,  for some $n_0$ large enough, $\lambda_+=\sup_{t\geq n_0}\frac{\lambda_+(t)}t$,  $\lambda_-=\inf_{t\geq n_0}\frac{\lambda_-(t)}t$.
In \cite{GLP} we use the notation $\hat\varpi=\frac{2\lambda_-}{\lambda_+}$, $\varpi'=\min\{1,\hat\varpi\}$.

Next, we introduce a parameter  $\vartheta>0$ which  measures the homogeneity. Let $J^s\phi_t$ be the stable Jacobian of the flow. Given $n_0\in\bN$, we assume, for all $t\geq n_0$, 
\begin{equation}\label{eq:homo}
\left[\sup_{x\in M}\ln J^s\phi_t(x)-\inf_{x\in M}\ln J^s\phi_t(x)\right]\leq \vartheta \lambda_-(t) d_s.
\end{equation}
With this notation we can state a correct version of \cite[Theorem 2.4]{GLP}.
\begin{thm}\label{thm:main2} For any $\cC^r, r > 2$, contact flow with
\[
\frac{\sqrt 5-1}{2}<\hat \varpi \;; \quad \vartheta<\frac{(\varpi')^2+\varpi'-1}{2d_s(1+\varpi')}  
\]
there exists $\tau_*>0$ such that the Ruelle zeta function is analytic in $\{z\in\bC\;:\;\Re(z)\geq \htop-\tau_*\}$ apart from a simple pole at $z=\htop$.\\
\end{thm}

The rest of this note contains the proof of Theorem \ref{thm:main2}.

First, it suffices to prove \cite[Proposition 7.5]{GLP} under the hypotheses of Theorem \ref{thm:main2}, since the derivation of \cite[Theorem 2.4]{GLP} from  \cite[Proposition 7.5]{GLP},  holds unchanged.

By \cite[Remark 7.1]{GLP}  we can restrict the discussion to $d_s=(d-1)/2$ forms. Since the first inequality in \cite[Proposition 7.5]{GLP} is correct, we need only prove the second. 
Also \cite[equation (7.6)]{GLP} is correct, hence it suffices to estimate the $\|\cdot\|_{1+\eta}$ norm of some power of ${\widehat R}_n(z)$. Indeed, by  \cite[equation (7.7) and the previous displayed equation]{GLP}, for each $z=a+ib$, $a\geq \sigma_{d_s}$, 
and $\eta\in [0,1]$,
\begin{equation}\label{eq:reg-up0}
\begin{split}
\pqnorm[\eta]{{\widehat R}_n(z)^3 h}^s\leq &\;\frac{C_\eta }{(a-\htop+\lambda\eta)^{n}(a-\htop)^{2n}}  \pqnorm[\eta]{h}^s\\
&+\frac{C_\eta}{(a-\htop)^{n}} \pqnorm[1+\eta]{{\widehat R}_{n}(z)^2h}^s.
\end{split}
\end{equation}
We will use the above equation instead of \cite[equation (7.7)]{GLP}.
\begin{rem}\label{rem:semplify_form}
The estimate \eqref{eq:reg-up0} can be restricted to forms proportional to the volume on the stable manifold.
More precisely, given a stable  manifold $W$, if $\{v^s_i\}_{i=1}^{d_s}$, $\{v^u_i\}_{i=1}^{d_s}$ are a base for the tangent space of $W$ and the unstable foliation, respectively, and $\{dx_i\}_{i=1}^{2ds+1}=\{dx^s_i, dx^u_i\}_{i=1}^{d_s}\cup\{dx_0\}$ the dual base ($dx_0$ being the flow direction), then for all $g$ not proportional to $w^s:=dx^s_1\wedge\cdots\wedge dx^s_{d_s}$ we have
\[
\left|\int_W\langle g , {\widehat R}_n(z)^3 h\rangle\right|\leq \;\frac{C_\eta }{(a-\htop+\lambda\eta)^{3n}}  \pqnorm[\eta]{h}^s,
\]
which yields already the required estimate.
Hence, from now on, by $\Gamma^{d_s,\alpha}_c(\widetilde W_+)$, defined in \cite[Section 3.2]{GLP}, we mean the subset of forms proportional to $w^s$.\\
\end{rem}
\begin{rem} \label{rem:unstable}
If $v\in \cV^u$,\footnote{ These are the unstable vector fields, see \cite[Definition 7.2]{GLP} for a precise definition.}   then the Lie derivative  $L_v$ acting on the above $d_s$ forms is well defined even for H\"older vector fields. Indeed, the pushforward by the flow generated by $v$ yields a quantity proportional to the Jacobian of the unstable holonomy which is well defined, together with its derivative along the unstable direction.
\end{rem}

Next, we must  estimate the right hand side of  \eqref{eq:reg-up0}: let $g\in\hat\Gamma^{d_s,1+\eta}_c$ and $h\in\Omega_{0,1}^{d_s}$,
\begin{equation}\label{eq:dolgo-step1}
\begin{split}
&\int_{W_{\alpha,G}} \langle g, {\widehat R}_n(z)^2 h\rangle =\sum_{k, \beta, \bi}\;\sum_{W\in \cW_{k, \beta,\bi}}\int_{\widetilde W}\langle\hat g_{k, \beta,\bi}, {\widehat R}_n(z) h\rangle\\
&\hat g_{k, \beta,\bi}=\varphi_{k, \beta,\bi}\frac{(k r+\tau_W)^{n-1}J_W\phi_{k r}\circ\phi_{\tau_W}}{e^{z(k r+\tau_W)}(n-1)!}  *\phi_{k r+\tau_W}^**g\\
&\varphi_{k, \beta,\bi}(x)=\psi_\beta(x)\Phi_{r,\bi}(\Theta_\beta(x))p(r^{-1}\tau_W(x))\|V(x)\|^{-1}.
\end{split}
\end{equation}
Recall that $\tau_W:\widetilde W\doteq \cup_{t\in[-2r,2r]}\phi_t W\to \bR$ is defined by $\phi_{-\tau_W(x)}(x)\in W$.\footnote{ The point of the above equation is that it allows one to go from an integral over a strong stable manifold to integrals over weak stable manifolds. See \cite[Section 3]{GLP} for the necessary definitions. To compare the formulae below with \cite[ equations (7.9, 7.10, 7.11)]{GLP} recall that the flow is contact, hence $J\phi_t=1$, and $(-1)^{d_s(d-d_s)}=(-1)^{d_s(d_s+1)}=1$. Also, recall that $\sum_{k\in\bZ}p(k+t)=1$ and $\operatorname{supp}(p)\subset (-1,1)$. Finally,  the minus sign in front of $z$ in \cite[equation (7.10)]{GLP} is a misprint and, just before  \cite[Equation (7.9)]{GLP}, the definition of $\tau_W$ has a minus sign missing due to a misprint.}

Next, as in \cite[ Equation (7.13)]{GLP}, we want to ``project"  $\hat g_{k, \beta, \bi}$ from $\widetilde W$ to some preferred manifold $\widetilde W_{k,\beta,\bi}$. To this end, we need a refinement of \cite[Lemma 7.3]{GLP}.

\begin{lem}\label{sublem:holo} For each $\alpha\in\cA$, $W,W'\in\Sigma_\alpha$ such that $H_{W,W'}(\widetilde W)\subset\widetilde W_+'$, $\vf\in \Gamma^{d_s, q}_c(\widetilde W)$, $q\in [0,1]$, supported in a ball of size $r$, there exists $\hat \vf\in \Gamma^{d_s,q\varpi'}_c(\widetilde W'_+)$, $\|\hat \vf\|_{\Gamma^{d_s,q\varpi'}_c(\widetilde W'_+)}\leq C_\#\|\vf\|_{\Gamma^{d_s,q}_c(\widetilde W)}$, such that for all $h\in\Omega_r^{d_s}$ we have
\[
\left|\int_{\widetilde W}\langle \vf, h\rangle-\int_{\widetilde W_+'}\langle \hat  \vf,h\rangle\right|\leq C_\# r d(W,W') \pqnorm[]{h}^u\|\vf\|_{\Gamma^{d_s,0}_c(\widetilde W)}.
\]
\end{lem}
\begin{proof}
Working in appropriate coordinates we can write $\widetilde W'_+$ as $\{(\xi,0,\tau)\}_{(\xi,\tau)\in\bR^{d_s+1}}$, and  $\widetilde W$ as  $\{(\xi,\be,\tau)\}_{(\xi,\tau)\in\bR^{d_s+1}}$ for $\be=d(W,W')e_{1}$.

We can describe the unstable foliation by $\bU(\xi,\eta,\tau)=(U(\xi,\eta), \eta, \Upsilon(\xi,\eta)+ \tau)$, $\bU(\xi,0,\tau)=(\xi,0,\tau)$.
Then the intersection between $\widetilde W_\xu=\{(\xi,\xu \be,\tau)\}_{(\xi,\tau)\in\bR^{d_s+1}}$ and the fiber $\bU(\xi,\cdot,\tau)$ gives the holonomy
$\bH_\xu(\xi,\tau)=(U(\xi,\xu \be),\xu \be, \Upsilon (\xi, \xu\be)+\tau)$. As mentioned in Remark \ref{rem:semplify_form}, $\vf=\bar\vf \, d\xi_1\wedge\cdots \wedge d\xi_{d_s}$, hence we can assume w.l.o.g. that $h=\bar h\, d\xi_1\wedge\cdots \wedge d\xi_{d_s}$, for some function $\bar h$.
It is then natural to define, for each $\xi\in\bR^{d_s}$, $\tau\in\bR$ and $\eta\in\bR^{d_s}$,
\[
\overline\bH_s(\xi,\eta,\tau)=\bU(\bU^{-1}(\xi,\eta, \tau)+(0,s\be,0)).
\]
Since, $\overline\bH_0(\xi,\eta, \tau)=(\xi,\eta, \tau)$ and $\overline\bH_r\circ \overline\bH_s=\overline\bH_{r+s}$, we have just defined a flow, let $\bar w=(w,\be,\sigma)$ be the associated vector field. Note that, by construction, $\bar w$ is a vector field in the unstable direction.  By the regularity of the holonomy (see the discussion at the beginning of \cite[Appendix E]{GLP}), we have $\|\bar w\|\leq \Const d(W,W')$. Hence, $\hat w=d(W,W')^{-1}\bar w\in \cV^u$.

Since $\overline \bH_s^* h=\bar h\circ \overline\bH_s\, J\overline\bH_s \, d\xi_1\wedge\cdots \wedge d\xi_{d_s}$,  $J\overline\bH_s$ is the Jacobian of $\overline\bH_s$, we have
\[
\begin{split}
\int_{\widetilde W}\langle \vf, h\rangle&=\int_{\widetilde W_+'}\langle \vf, h\rangle\circ \bH_1\cdot J\bH_1
=\int_{\widetilde W_+'}\bar \vf\circ \bH_1\bar h\circ\overline \bH_1\, J\overline\bH_1  \\
&=\int_{\widetilde W_+'}\int_0^1 \bar \vf\circ \bH_1\frac{d}{ds}\left(\bar h\circ\overline \bH_s\, J\overline\bH_s\right) ds+\int_{\widetilde W_+'}\bar \vf\circ \bH_1\bar h. 
\end{split}
\]
Since 
\[
\begin{split}
\frac{d}{ds}\left(\bar h\circ\overline \bH_s\, J\overline\bH_s\right) &=\frac{d}{ds}\langle d\xi_1\wedge\cdots \wedge d\xi_{d_s}, \overline \bH_s^* h\rangle=\langle d\xi_1\wedge\cdots \wedge d\xi_{d_s},\bH_s^* L_{\bar w}h\rangle\\
&=\langle d\xi_1\wedge\cdots \wedge d\xi_{d_s}, L_{\bar w}h\rangle\circ \overline\bH_s J\overline \bH_s,
\end{split}
\]
it is convenient to define, 
\begin{equation}\label{eq:barvf_def}
\begin{split}
&\hat\vf=\bar \vf\circ \bH_1(\xi,\tau)d\xi_1\wedge\cdots\wedge d\xi_{d_s}\\
&\psi_s=\bar\vf\circ \bH_1\circ \overline\bH_s^{-1}d\xi_1\wedge\cdots\wedge d\xi_{d_s},
\end{split}
\end{equation}
which, setting $\widetilde W_s=\{(\xi, s\be, \tau)\}_{(\xi,\tau)\in\bR^{d_s+1}}$, allows to write
\[
\begin{split}
\int_{\widetilde W}\langle \vf, h\rangle&=\int_{\widetilde W_+'}\int_0^1 \langle \psi_s, L_{\bar w}h\rangle\circ \overline\bH_s\, J\overline \bH_s ds+\int_{\widetilde W_+'} \langle \hat\vf, h\rangle \\
&=\int_{\widetilde W_+'}\int_0^1 \langle \psi_s, L_{\bar w}h\rangle\circ \bH_s\, J\bH_s ds+\int_{\widetilde W_+'} \langle \hat\vf, h\rangle \\
&=d(W,W')\int_0^1 ds\int_{\widetilde W_s}\langle \psi_s,   L_{\hat w}h\rangle+\int_{\widetilde W_+'} \langle \hat\vf, h\rangle .
\end{split}
\]
From the above equation the Lemma follows, since $\|\psi_s\|_{\Gamma_c^{d_s,0}(\widetilde W_s)}\leq \Const\|\vf\|_{\Gamma^{d_s,0}_c(\widetilde W)}$ and $\|\hat\vf\|_{\Gamma^{d_s,q\varpi'}_c(\widetilde W')}\leq \Const\|\vf\|_{\Gamma^{d_s,q}_c(\widetilde W)}$. The extra $r$ comes from the size of the support of $\vf$, and hence of $\psi_s$, in the flow direction.
\end{proof}

Next, following verbatim \cite{GLP}, and using Lemma \ref{sublem:holo} (with $q=1$), we obtain the equivalent of \cite[Equation (7.13)]{GLP}: for each $g\in \Gamma^{d_s,1+\eta}_c(W_{\alpha,G})$,
\begin{equation}\label{eq:dolgo-step2}
\begin{split}
\int_{W_{\alpha,G}} \langle g, {\widehat R}_n(z)^2 h\rangle =&\sum_{k, \beta, \bi}\sum_{W\in \cW_{k, \beta, \bi}}\int_{\widetilde W_{\beta, \bi}}\langle \hat g_{k, \beta, \bi, W}, {\widehat R}_n(z) h
\rangle\\
&+\sum_k\cO\left(\frac{ r^2 (kr)^{n-1}\pqnorm[]{\widehat R_n(z){h}}^u \|g\|_{\Gamma^{d_s,1+\eta}_c(W_{\alpha,G})}}{(n-1)!e^{(a-\sigma_{d_s})kr}}\right),
\end{split}
\end{equation}
where $\hat g_{k, \beta, \bi, W}$ is given by \eqref{eq:barvf_def}, with $\bH_1= \bH_{\widetilde W_{\beta,\bi},\widetilde W}$ being the holonomy between $\widetilde W_{\beta,\bi}$ and $\widetilde W$. Note that equation \eqref{eq:barvf_def} was implicitly used (but missing) in \cite{GLP}. 

Next, we slightly depart from \cite{GLP} insofar as we simplify immediately the expression of $\hat g_{k, \beta, \bi, W}$ instead of doing it during the proof of \cite[Lemma 7.10]{GLP}.

By \eqref{eq:dolgo-step1} we can write
\[
\hat g_{k, \beta, \bi}=\vf_{k, \beta, \bi}e^{-z\tau_W}\check g_{k, \beta, \bi}.
\]
Hence, setting $\vf_{k, \beta, \bi, W}=\vf_{k, \beta, \bi}\circ  \bH_{\widetilde W_{\beta,\bi},\widetilde W}$, by the first line of \eqref{eq:barvf_def} we have
\[
\hat g_{k, \beta, \bi, W}=\vf_{k, \beta, \bi, W}e^{-z\tau_W\circ \bH_{\widetilde W_{\beta,\bi},\widetilde W}}\check g_{k, \beta, \bi, W},
\]
where, recalling \cite[Equation (7.12)]{GLP},
\begin{equation}\label{eq:g1}
\begin{split}
&\|\check g_{k, \beta, \bi, W}\|_{\Gamma^{\varpi'}(\widetilde W_{\beta,\bi})}\leq \Const \frac{(k r)^{n-1}e^{-a k r}}{(n-1)!}\|g\|_{\Gamma^{d_s,1+\eta}_c(W)}\\
&\|\vf_{k, \beta, \bi, W}\|_{\Gamma^{\varpi'}(\widetilde W_{\beta,\bi})}\leq \Const r^{-1}.
\end{split}
\end{equation}
Hence, setting $\bc_{k, \beta, \bi, W}=\check g_{k, \beta, \bi, W}(x^{\bi})$, 
\begin{equation}\label{eq:g2}
\|\check g_{k, \beta, \bi, W}-\bc_{k, \beta, \bi, W}\|_{\Gamma^{0}_c(\widetilde W_{\beta,\bi})}\leq \Const r^{\varpi'} \frac{(k r)^{n-1}e^{-a k r}}{(n-1)!}\|g\|_{\Gamma^{d_s,1+\eta}_c(W)}.
\end{equation}

Next, setting $\Delta^*_W(\xi)=\tau_W\circ \bH_{\widetilde W_{\beta,\bi},\widetilde W}(\xi)-\xi_{2d_s+1}$ and $w_W(\xi)=\bH_{\widetilde W_{\beta,\bi},\widetilde W}(\xi)-\xi$,\footnote{ Here, again, we are computing using some appropriate coordinates.} we have that \cite[Equations (7.27) and (7.30) ]{GLP} implies, for all $\zeta=(\tilde \zeta,0)$ with $\|\tilde \zeta\|\leq r$,
\begin{equation}\label{eq:holder-temp}
\begin{split}
\left\| \Delta^*_W(\xi+\zeta)-\Delta^*_W(\xi)-d\alpha_0(w_W(x^{\bi}),\zeta)\right\|&\leq \Const r^{2-\varpi'}\|w_W(x^{\bi})\|^{\varpi'}\|\zeta\|^{\varpi'} \\
&\leq \Const r^{2+\varpi'}.
\end{split}
\end{equation}
We impose, for $\varsigma$ small enough,
\begin{equation}\label{eq:b-bound1}
|b|\leq r^{-2-\varpi'+\varsigma}\,,
\end{equation}
and define
\begin{equation}\label{eq:gclean}
\begin{split}
\mathfrak g_{k, \beta, \bi}&= \sum_{W\in \cW_{k, \beta, \bi}}\hat g_{k, \beta, \bi, W}\\
\bg_{k, \beta, \bi, W}(\xi)&=\vf_{k, \beta, \bi, W}(\xi)e^{-z d\alpha_0(w_W(x^{\bi}),\xi-x^\bi)}e^{-z(\Delta^*_W(x^\bi)+\xi_{2d_s+1})}\bc_{k, \beta, \bi, W} \\
&\doteq\vf_{k, \beta, \bi, W}(\xi)e^{-z d\alpha_0(w_W(x^{\bi}),\xi-x^\bi)}\hat\bc_{k, \beta, \bi, W}(\xi_{2d_s+1})\\
\mathfrak g^*_{k, \beta, \bi}&\doteq \sum_{W\in \cW_{k, \beta, \bi}}\bg_{k, \beta, \bi, W}\,.
\end{split}
\end{equation}
Letting  $D_{k,\beta,\bi} = \frac{(k r)^{n-1}e^{-a rk}}{(n-1)!}\#\cW_{k, \beta, \bi}$ and recalling \eqref{eq:holder-temp}, \eqref{eq:b-bound1} and \cite[Equation (7.30)]{GLP}, we can write, for $\varsigma\leq \varpi'$,
\begin{equation}\label{eq:celandiff}
\begin{split}
&\left\|\mathfrak g_{k, \beta, \bi}-\mathfrak g^*_{k, \beta, \bi}\right\|_{\Gamma_c^{0}(\widetilde W_{\beta,\bi})}\leq \Const r^{\varsigma} D_{k,\beta,\bi}\|g\|_{\Gamma^{d_s,1+\eta}} \\
&\left\|\mathfrak g_{k, \beta, \bi}\right\|_{\Gamma_c^{\varpi'}(\widetilde W_{\beta,\bi})}+\left\|\mathfrak g^*_{k, \beta, \bi}\right\|_{\Gamma_c^{\varpi'}(\widetilde W_{\beta,\bi})}
\leq \Const \left\{ \frac 1r+\frac{|b|}{r^{-2+\varpi'}}\right\} D_{k,\beta,\bi}
\|g\|_{\Gamma^{d_s,1+\eta}}.
\end{split}
\end{equation}
Note that, by the definition of $\widehat R(z)$ in \cite[Section 7.1]{GLP}, \cite[Equation (4.11)]{GLP}, \cite[Equation (4.17)]{GLP} and the related notation, for all $n\geq c_\star \ln r^{-1}$, with $c_\star$ large enough,
\begin{equation}\label{eq:clean-final}
\begin{split}
\sum_{k, \beta, \bi}&\left|\int_{\widetilde W_{\beta, \bi}}\langle[ \frkg_{k, \beta, \bi}- \frkg^*_{k, \beta, \bi}], {\widehat R}_n(z) h\rangle\right|
=\Bigg|\int_{c_a n}^\infty dt e^{-zt} \frac{t^{n-1}}{(n-1)!}\\
&\times \sum_{k, \beta, \bi}\sum_{\substack{\beta'\in\cA\\k'\in\widetilde K_\beta}}\int_{\widetilde W_{\beta', G_k'}}J_{W_{\beta', G_k'}}\phi_t\langle *\phi_t^**[ \frkg_{k, \beta, \bi}- \frkg^*_{k, \beta, \bi}], h\rangle\Bigg|\\
&\leq \Const\sum_{k, \beta, \bi}\int_{c_a n}^\infty dt e^{-at} \frac{t^{n-1}}{(n-1)!}\sum_{k, \beta, \bi}\sum_{\substack{\beta'\in\cA\\k'\in\widetilde K_\beta}} r^{1+\varsigma} D_{k,\beta,\bi}\|g\|_{\Gamma_c^{d_s,1+\eta}}\|h\|_{\eta}^s\\
&\leq \Const\int_{c_a n}^\infty\!\!\!\!  dt \int_{c_a n}^\infty \!\!\!\! ds\; e^{(\htop-a)(t+s)} \frac{t^{n-1}}{(n-1)!}\frac{s^{n-1}}{(n-1)!}r^{\varsigma}\|g\|_{\Gamma_c^{d_s,1+\eta}}\|h\|_{\eta}^s\\
&\leq \Const (a-\htop)^{-2n}r^{\varsigma}\|g\|_{\Gamma_c^{d_s,1+\eta}}\|h\|_{\eta}^s.
\end{split}
\end{equation}

Hence, by \eqref{eq:dolgo-step2},  \eqref{eq:clean-final} and \cite[Equation (7.6)]{GLP}, we can write
\begin{equation}\label{eq:dolgo-step3}
\begin{split}
\int_{W_{\alpha,G}} \langle g, {\widehat R}_n(z)^2 h\rangle =&\sum_{k, \beta, \bi}\int_{\widetilde W_{\beta, \bi}}\langle \frkg^*_{k, \beta, \bi}, {\widehat R}_n(z) h
\rangle+\cO\left(\frac{r^{\varsigma}\|g\|_{\Gamma_c^{d_s,1+\eta}}}{(a-\htop)^{2n}}\|h\|_{\eta}^s\right)\\
&+\cO\left(\frac{\pqnorm[\eta]{h}^u \|g\|_{\Gamma^{d_s,1+\eta}_c}}{(a-\htop+\bar\lambda)^{n}(a-\htop)^n}\right).
\end{split}
\end{equation}

To estimate the integral on the right hand side of \eqref{eq:dolgo-step3} we define, similarly to \cite{GLP}:
\[
\begin{split}
&\mathfrak G_{k, \beta, \bi, A}^*\doteq \sum_{W\in \cW_{k, \beta, \bi}}\sum_{W'\in A_{k, \beta, \bi}(W)}\langle\bg_{k, \beta, \bi, W}, \bg_{k, \beta, \bi, W'}\rangle\\
&\mathfrak G_{k, \beta, \bi, B}^*\doteq \sum_{W\in \cW_{k, \beta, \bi}}\sum_{W'\in B_{k, \beta, \bi}(W)}\langle\bg_{k, \beta, \bi, W},\bg_{k, \beta, \bi, W'}\rangle.
\end{split}
\] 

To conclude, we need Lemmata \ref{lem:dolgo79} and \ref{sublem:doest-3} which are refinements of \cite[Lemma 7.9]{GLP} and \cite[Lemma 7.10]{GLP}, respectively. The proof of Lemma \ref{sublem:doest-3} follows closely \cite[Lemma 7.10]{GLP}, but it applies the same logic to different objects. Conversely, Lemma \ref{lem:dolgo79} differs from  \cite[Lemma 7.9]{GLP} as we take advantage of our new homogeneity hypothesis \eqref{eq:homo}.
\begin{lem} \label{lem:dolgo79}
If   $c_a\geq n_0$ and $C_\# |b|^{-\frac{\bar\lambda\Cnd\Cnc}{2ea_0} }\leq \varrho \leq \const r^{\frac{1+\varsigma/d_s}{1-\vartheta} }$, for some $\varsigma>0$, and $\vartheta\in(0, 1)$, we have
\begin{equation}\label{eq:doest-2}
\begin{split}
\|\mathfrak G_{k, \beta, \bi, A}\|_\infty\leq& C_\# D_{k,\beta,\bi}^2 r^{\varsigma}\|g\|^2_{\Gamma^{d_s,1+\eta}_c(W)}.
\end{split}
\end{equation}
\end{lem}
\begin{proof}
The lower bound and the fact that the upper bound is bounded by the ratio between the volume of $\phi_{t}(D^u_{r}(W))$ and $\phi_{t}(D^u_{\varrho}(W))$ is proven exactly as in \cite[Lemma 7.9]{GLP}. The novelty here consists in a different estimate of such a ratio.\\
Let $t_0\in\bN$ be such that $e^{\lambda_-(t_0)}\varrho=1$. Then, for  each $x\in D^u_{\varrho}(W)$, let  $B(x)$ be an unstable disc of diameter $1$ and centred at $\phi_{t_0}(x)$, clearly $B(x)\subset \phi_{t_0}(D^u_{2\varrho}(W))$. Thus we can cover $\phi_{t_0}(D^u_{\varrho}(W))$ with $N_\varrho=\Const |\phi_{t_0}(D^u_{\varrho}(W)|$ discs. On the other hand, arguing analogously, we can find $N_r=\Const |\phi_{t_0}(D^u_{r}(W)|$ disjoint unstable discs of diameter $1$ contained in $\phi_{t_0}(D^u_{r}(W)$. 

By \eqref{eq:homo} and since the flow is contact, setting  $J_-^u:=\inf_{x\in M}J^u\phi_{t_0}(x)$, we have 
\[
\begin{split}
\frac{N_r}{N_\varrho}\geq\const \frac{|\phi_{t_0}(D^u_{r}(W)|}{|\phi_{t_0}(D^u_{\varrho}(W)|}\geq \const \frac{J_-^u r^{d_s}}{J_-^u e^{\vartheta\lambda_- (t_0) d_s}\varrho^{d_s}}=\const \left(r\varrho^{\vartheta-1}\right)^{d_s}\geq \Const r^{-\varsigma}.
\end{split}
\]
On the other hand by \cite[Lemmata C.1, C.3]{GLP} we have that all the discs of radius one grow under the dynamics at the same rate (given by the topological entropy), hence for all $t\geq t_0$, we have the required estimate
\[
\frac{|\phi_{t}(D^u_{r}(W)|}{|\phi_{t}(D^u_{\varrho}(W)|}\geq  \Const r^{-\varsigma}.
\]
\end{proof}

\begin{lem}\label{sublem:doest-3}
For $|b|\leq r^{-2-\varpi'+\varsigma}$ we have
\begin{equation}\label{eq:doest-3}
\left|\int _{\widetilde W_{\beta, \bi}} \mathfrak G_{k, \beta, \bi, B}^*\right|\leq C_\# |b|^{-{\varpi'}} \varrho^{-{\varpi'}}r^{d_s}D_{k,\beta,\bi}^2\|g\|^2_{\Gamma^{d_s,1+\eta}_c(W)}.
\end{equation}
\end{lem}
\begin{proof}
For future convenience let us set $(\eta^s_{W,W',\bi},\eta^u_{W,W',\bi},\eta^d_{W,W',\bi})=\eta_{W,W',\bi}=w_W(x^{\bi})-w_{W'}(x^{\bi})$ and $\eta^+_{W,W',\bi}=w_W(x^{\bi})+w_{W'}(x^{\bi})$. By assumption $\|\eta_{W,W',\bi}\|\geq \varrho$. Also, it is convenient to work in coordinates $(\xi,\eta,\tau)$, $\xi,\eta\in\bR^{d_s}$, in which $x_\bi=0$ and $W_{\beta,\bi}\subset \{(\xi,0)\;:\; \xi\in \bR^{d_s}\}$ and $d\alpha_0=\sum_{i=1}^{d_s} d\xi_i\wedge d\eta_i$. We must estimate
\begin{equation}\label{eq:doest-4}
\begin{split}
\int_{\widetilde W_{\beta, \bi}}\hskip-.2cm&\langle \bg_{k, \beta, \bi,W},\bg_{k, \beta, \bi,W'}\rangle
=\int_{\widetilde W_{\beta, \bi}}\vf_{k, \beta, \bi, W}(\xi)\vf_{k, \beta, \bi, W'}(\xi)\\
&\times \hat\bc_{k, \beta, \bi, W}(\tau)\overline{\hat\bc_{k, \beta, \bi, W'}(\tau)} e^{-ib d\alpha_0(\eta_{W,W',\bi},\xi-x^\bi)}e^{-a d\alpha_0(\eta^+_{W,W',\bi},\xi-x^\bi)}\, .
\end{split}
\end{equation}
As in \cite[Section 7.2]{GLP} we choose $y_{W,W'}=(-\eta^u_{W,W',\bi}\|\eta^u_{W,W',\bi}\|^{-1},0,0)$ which implies that $d\alpha_0(\eta_{W,W',\bi}, y_{W,W'})=\|\eta^u_{W,W',\bi}\|$, and $\langle y_{W,W'}, e_{2d_s+1}\rangle=0$. Also, let $\Sigma_W=\{(\xi, \tau)\in\bR^{d_s+1}\;|\:\;\langle \xi, \eta^u_{W,W',\bi}\rangle=0\}$ and 
\[
\begin{split}
A(\xi,s,\tau)=&\vf_{k, \beta, \bi, W}((\xi,0,\tau )+sy_{W,W'})\vf_{k, \beta, \bi, W'}((\xi, 0,\tau)+sy_{W,W'})\\
&\times e^{-a d\alpha_0(\eta^+_{W,W',\bi},(\xi,s,0,\tau)+sy_{W,W'})}.
\end{split}
\]
Then, we can write
\[
\int_{\widetilde W_{\beta, \bi}}\hskip-.2cm\langle \bg_{k, \beta, \bi,W},\bg_{k, \beta, \bi,W'}\rangle
=\int_{\Sigma_W}\!\!\!\! \!\!\! d\tau d\xi\, \hat\bc_{k, \beta, \bi, W}(\tau)\overline{\hat\bc_{k, \beta, \bi, W'}(\tau)} \int_{-\const r}^{\const r} \hskip-12pt ds A(\xi, s,\tau)e^{-ib \|\eta^u_{W,W',\bi}\|s }.
\]
Note that, by \eqref{eq:g1}, $\|A(\xi, \cdot,\tau)\|_{\cC^{\varpi'}}\leq\Const  r^{-1}$, hence (as in \cite[Lemma 7.10]{GLP})
\[
\left| \int_{-\const r}^{\const r} \hskip-12pt ds A(\xi, s)e^{-ib \|\eta^u_{W,W',\bi}\|s }\right|\leq \Const |b|^{-{\varpi'}} \varrho^{-{\varpi'}}.
\]
\vskip-.8cm
\end{proof}
\vskip.2cm
Here our strategy departs from \cite{GLP} as we control directly  where $\frkg^*_{k,\beta,\bi}$ is large.

Let  $\Omega=\{x\in  W_{\beta,\bi}\;:\;  \|\frkg^*_{k,\beta,\bi}(x)\|\geq 4 C_\flat |r|^{{\varsigma}/2}D_{k,\beta,\bi}\}$ and $\Omega_1=\{x\in  W_{\beta,\bi}\;:\; \mathfrak G_{k, \beta, \bi, B}^*\geq C_\flat |r|^{{\varsigma}}D_{k,\beta,\bi}^2\}$, while $\widetilde \Omega$ and $\widetilde \Omega_1$ are $r$ thickenings in the flow direction.  Note that if $x\in\Omega$, then $\phi_t(x)\in\widetilde\Omega_1$ for all $t\leq \Const r$. By Lemma \ref{lem:dolgo79}, choosing  $C_\flat$ large, we have $\Omega\subset \Omega_1$. By Chebychev inequality, Lemma \ref{sublem:doest-3} implies
\[
\begin{split}
|\widetilde\Omega|&= \int_{\widetilde W_{\beta,\bi}}\Id_{\Omega_1}\leq \Const \int_{\widetilde W_{\beta,\bi}}
\mathfrak G_{k, \beta, \bi, B}^*|r|^{-{\varsigma}}D_{k,\beta,\bi}^{-2}
\leq \Const (|b| \varrho)^{-{\varpi'}}r^{d_s-{\varsigma}}.
\end{split}
\]
Thus 
\begin{equation}\label{eq:Omegaest}
|\Omega|\leq \Const |b|^{-{\varpi'}} \varrho^{-{\varpi'}}r^{d_s-1-{\varsigma}} .
\end{equation}

If $x\in\Omega$ then, by \eqref{eq:gclean}, \eqref{eq:celandiff}, we have that
$\|\frkg^*_{k,\beta,\bi}(y)\|\geq C_\flat |r|^{{\varsigma}/2}D_{k,\beta,\bi}$ provided 
\[
|y-x|^{\varpi'}r^{-1}+|y-x| |b| r\leq \Const r^{{\varsigma}/2}.
\]
The above holds for 
\begin{equation}\label{eq:Deltabound}
|y-x|\leq \Const \min\{ r^{\frac{1+\varsigma/2}{\varpi'}},|b|^{-1}r^{-1+\frac{{\varsigma}}2}\}=:\rho.
\end{equation}

We are finally ready to prove the stated  Theorem. 
\begin{proof}[\bf Proof of Theorem \ref{thm:main2}] 
Let $t_0>0$ be such that $e^{\lambda_- (t_0)}\rho=1$. 
Then, recalling \eqref{eq:homo}, 
\[
\begin{split}
&|\phi_{-t_0}(\Omega)|\leq e^{(J^s_-(t_0)+\lambda_- (t_0)d_s\vartheta)}|\Omega|\\
&|\phi_{-t_0}(W_{\beta,\bi})|\geq e^{J^s_-(t_0)} r^{d_s}.
\end{split}
\]
It follows that if we cover $\phi_{-t_0}(W_{\beta,\bi})$ by discs of radius $1$, then, recalling \eqref{eq:Omegaest}, for each disc that intersects $\phi_{-t_0}(\Omega)$ there are at least
\[
K=\frac{r^{d_s}}{ e^{\lambda_- (t_0)d_s\vartheta }|\Omega|}\geq \const \rho^{\vartheta d_s}b^{\varpi'}\varrho^{\varpi'}r^{1+\varsigma}
\]
discs that are disjoint from $\phi_{-t_0}(\Omega)$. Indeed, if a disc intersects $\phi_{-t_0}(\Omega)$, then a disc twice its radius must have a fixed proportion of its volume belonging to $\phi_{-t_0}(\Omega)$.\\
We chose $\varrho = \const r^{\frac{1+\varsigma/d_s}{1-\vartheta} }$ (so Lemma \ref{lem:dolgo79} applies), $|b|=r^{-2-\varpi'+\varsigma}$ (so Lemma \ref{sublem:doest-3} applies). Accordingly, $K$ can be larger than one only if $(\varpi')^2+\varpi'-1>0$, but then, choosing $\varsigma$ small enough, \eqref{eq:Deltabound} implies $\rho=\Const r^{1+\varpi'-\frac{{\varsigma}}2}$, which implies
\[
K\geq r^{-\varsigma}
\]
provided $\vartheta<\frac{(\varpi')^2+\varpi'-1}{2d_s(1+\varpi')}$ and $\varsigma$ is small enough. Again by \cite[Lemmata C.1, C.3]{GLP} this ratio persists under iteration. Hence, for each $t\geq t_0$,
\[
|\phi_{-t}(\Omega)|\leq \Const r^{\varsigma} |\phi_{-t}( W_{\beta,\bi})|.
\]
If $n\geq C_1\ln|b|$,  for $C_1$ and $b$ large enough, by \cite[Equation (7.12)]{GLP},  $c_a n\geq t_0$ and
\[
\|J_{W}\phi_{-t_0} *\phi_{-t_0}^** \mathfrak g^*_{k, \beta, \bi}\|_{\Gamma^{d_s,\varpi'}_c(W)}
\leq \Const\|  \mathfrak g^*_{k, \beta, \bi} \Id_{\phi_{t_0}(W)} \|_{\Gamma^{d_s,0}_c(W_{\beta,\bi})}.
\]
Hence, for each $k'\geq c_a n$, we can decompose $\phi_{-k'} W_{\beta,\bi}=\cup_{W\in \cW_{k'}} W$, $W\in\Sigma^s$ (see \cite[Definition 7.2]{GLP}). 
We write $\cW_{k'}=\cW_{k'}^+\cup \cW_{k'}^-$, where $W\in \cW_{k'}^-$ if $\phi_{k'}(W)\cap \Omega=\emptyset$, while if $W\in \cW_{k'}^+$, then $\phi_{k'}(W)\subset \Omega_1$. By the previous discussion, and \cite[Lemmata C.1, C.3]{GLP}, we have $\sharp \cW_{k'}^+\leq \Const r^{\varsigma}\sharp \cW_{k'}$.
Since,
\[
\left|\int_{\widetilde W_{\beta, \bi}}\langle \frkg^*_{k, \beta, \bi}, {\widehat R}_n(z) h \rangle\right|\leq \sum_{k'}\;\sum_{W\in\cW_{k'}}\left|\int_{\widetilde W}\langle {\boldsymbol \vf}_{k', W}, h\rangle\right|
\] 
with 
$\|  {\boldsymbol \vf}_{k', W}\|_{\Gamma^{d_s,\varpi'}_c(W)}\leq \Const \frac{(k'r)^{n-1}e^{a k'r}}{(n-1)!}\|  \frkg^*_{k, \beta, \bi}\Id_{\phi_{k'}(W)}\|_{\Gamma^{d_s,0}_c(W_{\beta,\bi})}$, it follows,
\[
\begin{split}
&\sum_{k,\beta,\bi}\left|\int_{\widetilde W_{\beta, \bi}}\!\!\!\!\langle \frkg^*_{k, \beta, \bi}, {\widehat R}_n(z) h \rangle\right|\leq \Const\!\sum_{k,\beta,\bi, k'} \frac{(k'r)^{n-1}\left[\sharp \cW_{k'}^-  |r|^{\varsigma/2}+\sharp \cW_{k'}^+\right]D_{k,\beta,\bi}}{e^{-a k'r}(n-1)!}\|h\|^*_{\varpi'}\\
&\leq \Const r^{\varsigma/2}\sum_{k, k'} \frac{(k r)^{n-1}e^{-a rk}}{(n-1)!} \frac{(k'r)^{n-1}|\phi_{k+k'}(W_{\alpha, G})}{e^{-a k'r}(n-1)!} \|h\|^*_{\varpi'}\\
&\leq \Const r^{\varsigma/2} (a-\htop)^{-2n}\|h\|^*_{\varpi'}.
\end{split}
\]
Using the above inequality in \eqref{eq:dolgo-step3} provides an estimate of $\|\widehat R_n(z)^2 h\|^s_{1+\eta}$ which, substituted into  \eqref{eq:reg-up0}, yields \cite[Proposition 7.5]{GLP}.\\
 Theorem \ref{thm:main2} follows then as in \cite[Theorem 2.4]{GLP}.
\end{proof}


\begin{thebibliography}{GLP}
\bibitem{GLP} Giulietti, P.; Liverani, C.; Pollicott, M. {\em Anosov flows and dynamical zeta functions}. Ann. of Math. (2) {\bf 178} (2013), no. 2, 687--773.
\end{thebibliography}
\end{document}